\documentclass[a4paper,notitlepage]{article}

\newcommand{\thetitle}{
  Schr\"odinger Operators on Regular Metric Trees with
  Long Range Potentials: Weak Coupling Behavior
}
\newcommand{\theabstract}{
  Consider a regular $d$-dimensional metric tree $\Gamma$ with root $o$.
  Define the Schr\"odinger
  operator $-\Delta - V$, where $V$ is a non-negative,
  symmetric potential, on $\Gamma$, with Neumann boundary conditions
  at $o$. Provided that $V$ decays like $x^{-\gamma}$ at infinity,
  where $1 < \gamma \leq d \leq 2, \gamma \neq 2$, we will determine
  the weak coupling behavior of the bottom of the spectrum of $-\Delta - V$.
  In other words, we will describe
  the asymptotical behavior of
  $\inf \sigma(-\Delta - \alpha V)$ as $\alpha \to 0+$.
}
\newcommand{\thekeywords}{
  Schr\"odinger operators, metric trees, Fourier-Bessel
  transformation, weak coupling.  
}

\title{\thetitle}

\author{
  Tomas Ekholm\\
  {\small Department of Mathematics}\\
  {\small Lund University}\\
  {\small S-221 00 Lund, Sweden}\\
  {\small \texttt{tomas.ekholm@math.lu.se}}\\
  \and
  Andreas Enblom\\
  {\small Department of Mathematics}\\
  {\small Royal Institute of Technology}\\
  {\small S-100 44 Stockholm, Sweden}\\
  {\small \texttt{enblom@math.kth.se}}\\
  \and
  Hynek Kova\v{r}{\'\i}k\\
  {\small Dipartimento di Matematica}\\
  {\small Universit\`a degli Studi di Modena e Reggio Emilia}\\
  {\small I-411 00 Modena, Italy}\\
  {\small \texttt{hynek.kovarik@unimore.it}}\\
}

\usepackage{amsmath}
\usepackage{amsthm}
\usepackage{amssymb}     
\usepackage{mathrsfs}    
\usepackage{enumerate}

\theoremstyle{plain}
\newtheorem{thm}{Theorem}[section]
\newtheorem{lem}[thm]{Lemma}
\newtheorem{prop}[thm]{Proposition}

\theoremstyle{definition}
\newtheorem{defn}[thm]{Definition}
\newtheorem{rem}[thm]{Remark}

\newcommand{\st}{\,;\,}

\newcommand{\coloneq}{\mathrel{\mathop:}=}

\newcommand{\V}{\ensuremath{\mathcal{V}}}
\newcommand{\E}{\ensuremath{\mathcal{E}}}

\DeclareMathOperator{\impart}{Im}

\DeclareMathOperator{\supp}{supp}
\DeclareMathOperator{\tr}{Tr}

\newcommand{\utrans}{\ensuremath{\mathscr{U}}}
\newcommand{\vtrans}{\ensuremath{\mathscr{V}}}
\newcommand{\wtrans}{\ensuremath{\mathscr{W}}}

\newcommand{\reals}{\ensuremath{\mathbb{R}}}
\newcommand{\complex}{\ensuremath{\mathbb{C}}}
\newcommand{\hilbert}{\ensuremath{\mathscr{H}}}
\newcommand{\lspace}[2]{\ensuremath{L^{#1}\left(#2\right)}}
\newcommand{\sobspace}[2]{\ensuremath{H^{#1}\left(#2\right)}}

\newcommand{\cspace}[2]{\ensuremath{C^{#1}\left(#2\right)}}

\newcommand{\ltwo}{\lspace{2}{0, \infty}}

\newcommand{\cinf}{\cspace{\infty}{0, \infty}}

\newcommand{\spec}[1]{\ensuremath{\sigma\left(#1\right)}}
\newcommand{\dom}[1]{\ensuremath{D\left(#1\right)}}
\newcommand{\qdom}[1]{\ensuremath{D\left[#1\right]}}

\begin{document}


%
%
%
%
%
%
%



\setlength{\parindent}{0pt}
\setlength{\parskip}{1ex plus 0.5ex minus 0.2ex}

\maketitle

\begin{abstract}
  \theabstract

  \vspace{1em}
\end{abstract}


\section{Introduction}
It is a well known fact that the weak coupling asymptotics for a
Schr\"odinger operator $-\Delta -\alpha V$ in $\reals^n$ depends
on the dimension $n$ of the underlying space,
\cite{simon,bgs}. In case $n=1$ it has been shown in
\cite{simon,bgs} that if an integrable potential $V$ decays at
infinity faster than $x^{-2}$, then for $\alpha$ small enough the operator
$-\Delta -\alpha V$ has a unique eigenvalue $E_1(\alpha)$ which
satisfies the asymptotic equation
\begin{equation} \label{simon}
E_1(\alpha) \, \sim \, \left( \alpha \int_\reals V\right)^2,\quad \alpha\to 0 .
\end{equation}
Note that as long as $V$ satisfies the above criteria then the behavior of
$E_1(\alpha)$ is uniform in order of $\alpha$, i.e.~proportional to $\alpha^2$,
and the potential $V$ enters only in the coefficient.

If $V$ decays more slowly than $x^{-2}$ at infinity, the picture is
different,
since the operator $-\Delta -\alpha V$ now has
infinitely many eigenvalues for any $\alpha>0$.
However, Klaus has shown in \cite{klaus} that the ground state $E_1(\alpha)$
has different asymptotics than the rest of the discrete spectrum
but is still proportional to
$\alpha^2$, provided $V$ decays faster than $x^{-1}$. Finally,
when $V(x) \sim x^{-\gamma}$ with
$\gamma \leq 1$, then the quadratic dependence on
$\alpha$ fails to hold and the
corresponding power of $\alpha$ in the asymptotics of $E_1(\alpha)$ is fully
determined by the parameter $\gamma$, see \cite{klaus} for details.

In the present paper the goal is to study these questions in the
case of Schr\"odin\-ger operators defined on metric trees. Such
trees represent a particular case of the so-called quantum graphs,
that provide mathematical models for nano-technological devices
consisting of connected thin strips. Spectral analysis of Laplace
and Schr\"odinger operators on these structures has therefore
attracted a lot of attention, see
e.g.~\cite{carl,exner,kuch,ks,naimark-solomyak,sob-sol,solomyak}.

A metric tree $\Gamma$ consists of a set of vertices
and a set of edges (branches), the edges being one dimensional intervals
connecting the vertices; see Section \ref{trees} for details. In
\cite{kov} it was proved that the behavior of $E_1(\alpha)$ then
depends on the global structure of the tree. More precisely, if $V$
decays fast enough, then the asymptotic behavior of $E_1(\alpha)$ is
again uniform and given by
\begin{equation} \label{tree-basic}
E_1(\alpha) \, \sim \, \left( \alpha \int_{\Gamma} V\right)^{\frac {2}{2-d}},
\quad \alpha\to 0, \quad 1\leq d <2.
\end{equation}
where $d$ is the so-called dimension of the tree. Roughly speaking,
the value of $d$ tells us how fast the number of branches of the
tree increases as a function of the distance from its root, see
Section \ref{trees} for a precise definition.

In this work we will study the interplay between the global
structure of the tree, i.e.~its dimension d, and the decay of $V$.
In particular, it will be shown that if $V$ decays as $x^{-\gamma}$
with $1 < \gamma \leq d \leq 2, \gamma \neq 2$, then the
corresponding asymptotics of $E_1(\alpha)$ is again fully determined
by the behavior of $V$ at infinity, that is, by the value of
$\gamma$, see Theorem \ref{thm:main}. It is easily seen that such
potentials are not integrable on $\Gamma$, which means that the
method of \cite{simon, bgs,kov} cannot be applied and a different
approach is needed. In order to prove the main result we proceed in
several steps, which we now briefly outline.

In Section \ref{reduction} we show that $E_1(\alpha)$ is bounded
from above and from below by the infimum of the spectra of certain
auxiliary operators $H_\alpha^-$ and $H_\alpha^+$ which act in
a weighted $L^2$ space on the half-line $(0,\infty)$; see Lemma
\ref{lem:infspecsame}. This enables us to reduce the problem of the
asymptotical behavior of the ground state of $-\Delta -\alpha V$ on
$\Gamma$ to the problem of finding the asymptotical behavior of the
ground state of certain Schr\"odinger operators in weighted
$L^2$ spaces on the half-line with a weight which
depends parametrically on $d$. These operators are
analysed in Section \ref{sec:bottomspectrum}; see Theorem
\ref{thm:estimates}.

The main technical tools used in the paper are the Birman-Schwinger
principle and the Fourier-Bessel transformation. In
fact, the derivation of the latter transform may be considered
interesting on its own merits. It provides an explicit
formula for the unitary operator which transforms the
one-dimensional Laplace operator on $\lspace 2 {(0,\infty), (1+t)^a dt},\,
a \geq 0$ to a multiplication operator on $\ltwo$. We
therefore believe that it might be of a general interest also for
other applications. For the convenience of the reader, this material
together with the Weyl-Titchmarsch-Kodaira Theorem is described in
the Appendices.

\section{Main Result}

For functions $f,g: (0,\infty) \to \reals$,
we will use the notation $f(x) \approx g(x)$ to mean that
there are constants $C_1, C_2 > 0$ such that $C_1g(x) \leq f(x) \leq
C_2g(x)$ for any $x > 0$,
and the notation $f(x) \asymp g(x), x \to 0+$ to mean that there are
constants $D_1,D_2 > 0$ such that $D_1g(x) < f(x) < D_2g(x)$ for every
sufficiently small $x > 0$.

Having introduced the necessary notation we are in a position to
state the main result of the paper:

\begin{thm} \label{thm:main}
  Let $1 < \gamma \leq d \leq 2$, $\gamma \neq 2$. Suppose that
  $\Gamma$ is a regular metric tree of dimension $d$. Define the
  Neumann Laplacian $-\Delta$ in $\lspace 2 \Gamma$.
  Suppose that $V$ is non-negative, measurable, and such that
  \begin{displaymath}
    V(t) \approx \frac{1}{(1+t)^\gamma}.
  \end{displaymath}
  Let $\tilde V : \Gamma \to \complex$ be defined by $\tilde V(x) =
  V(|x|)$ for $x \in \Gamma$. Then we have:
  \begin{enumerate}[{\normalfont (i)}]
  \item If $1 < \gamma < d \leq 2$, then
    \begin{displaymath}
      \inf \spec{-\Delta - \alpha \tilde V}
      \asymp - \alpha^{\frac{2}{2-\gamma}},
      \quad \alpha \to 0+.
    \end{displaymath}
  \item If $1 < \gamma = d < 2$, then
    \begin{displaymath}
      \inf \spec{-\Delta - \alpha \tilde V}
      \asymp - \left|\alpha \log \alpha\right|^{\frac{2}{2-\gamma}},
      \quad \alpha \to 0+.
    \end{displaymath}
  \end{enumerate}
\end{thm}

\begin{rem}
Note that in contrast to the case treated in \cite{kov},
the asymptotics of $\inf \spec{-\Delta - \alpha \tilde V}$, for
$\gamma < d$, are
independent of $d$ and fully determined by the parameter $\gamma$. In
other words, they are determined by the behavior of $V$ at infinity.
This is analogous to the
regime $\gamma<1$ for one-dimensional Schr\"odinger operators, \cite{klaus}.

Moreover, in the border-line case $\gamma=d$ there is a logarithmic
correction to the power-like law. This is again reminiscent of the
behavior of one-dimensional Schr\"odinger operators
in the case where $V(x) \sim 1/(1+|x|)$, corresponding to $\gamma=1$;
see \cite{klaus}.
\end{rem}

As mentioned in the introduction, the proof of Theorem \ref{thm:main}
will proceed in several steps. The main result below is
Theorem \ref{thm:estimates}, that together with Lemma
\ref{lem:infspecsame} proves the main theorem.
We start with the preliminary section
introducing Schr\"odinger operators on regular metric trees.

\section{Spectral Theory on Regular Metric Trees}

The geometry of regular metric trees and the definition of the
Laplacian on those trees are discussed thoroughly in
\cite{solomyak, naimark-solomyak}.
We
give here a brief summary that serves our purposes.

\subsection{Regular Metric Trees}
\label{trees}
Let $\Gamma \subset \reals^2$ be a rooted metric tree with root $o$.
Let $\preceq$ be the natural partial ordering on $\Gamma$ defined by
letting $x \preceq y$ mean that either $x = y$ or that $x$ is on the
path from $o$ to $y$. We use the notation $x \prec y$ to mean that $x
\preceq y$ and $x \neq y$. For a point $x \in \Gamma$, let us denote by
$|x|$ the length of the path in $\Gamma$ from $o$ to $x$.

Let $\V = \V(\Gamma)$ be the set of vertices in $\Gamma$ and $\E =
\E(\Gamma)$ the set of edges in $\Gamma$. Clearly $\V \subset \Gamma$
and $e \subset \Gamma$ for each $e \in \E$. For simplicity we
consider each edge $e \in \E$ to be a straight line segment of
positive length, containing its endpoints. Let $v_1(e)$ and $v_2(e)$
be the endpoints of the edge $e$, ordered such that $v_1(e) \prec
v_2(e)$. Hence, for $e \in \E(\Gamma)$,
\begin{displaymath}
  e = \left\{
    x \in \Gamma \st v_1(e) \preceq x \preceq v_2(e)
  \right\}
  = \left\{ (1-t)v_1(e) + tv_2(e) \st 0 \leq t \leq 1\right\}.
\end{displaymath}

For any point $x \in \Gamma$, let $b(x)$ be its
\emph{branching number} definied by
\begin{displaymath}
  b(x) = \lim_{\epsilon \to 0+}
  \#
  \left\{y \in \Gamma \st
    y \succeq x, |y| = |x| + \epsilon
  \right\}.
\end{displaymath}
To clarify, if $x$ is a vertex of $\Gamma$, then $b(x)$ is the
number of edges emanating from $x$, and if $x$ is not a vertex, then
$b(x) = 1$. Naturally, we will assume throughout that $\Gamma$ is
such that $b(x) > 1$ for any vertex $x \neq o$, and that $b(o) = 1$. 
In other words,
there are no vertices, except $o$, that have only one emanating
edge.

\begin{defn}
  The tree $\Gamma$ is said to be \emph{regular} if $b(x) = b(y)$ for
  all points $x,y \in \Gamma$ satisfying $|x| = |y|$.
\end{defn}

Now we define the branching function $g_\Gamma$ by
\begin{displaymath}
  g_\Gamma(t) = \# \left\{
    x \in \Gamma \st |x| = t
  \right\},
  \quad t \geq 0.
\end{displaymath}
\begin{defn}
  The tree $\Gamma$ is said to have \emph{dimension} $d$ if there are
  constants $C_1, C_2 > 0$ such that
  \begin{displaymath}
    C_1(1+t)^{d-1} \leq g_\Gamma(t) \leq C_2(1+t)^{d-1}, \quad t \geq 0.
  \end{displaymath}
\end{defn}

Often one studies the \emph{height} $h(\Gamma)$ and the
\emph{reduced height} $L(\Gamma)$ of $\Gamma$ defined respectively
by
\begin{displaymath}
  h(\Gamma) = \sup_{x \in \Gamma} |x|
  \quad \textrm{ and } \quad
  L(\Gamma) = \int_0^{h(\Gamma)} \frac{dt}{g_\Gamma(t)}.
\end{displaymath}
We will be interested in regular $d$-dimensional trees of infinite reduced
height, for which $g_\Gamma$ is growing. Given the above
definitions, this means that $\Gamma$ is of infinite height and that
$1 < d \leq 2$.

\subsection{The Laplace and Schr\"odinger operators on $\Gamma$}

Here, $\Gamma$ will be a fixed, regular, $d$-dimensional metric tree,
where $1 < d \leq 2$. Define the Hilbert space
$\lspace 2 \Gamma$ as the space of
functions $f : \Gamma \to \complex$ satisfying that
$f \upharpoonright e \in \lspace 2 e$ for any $e \in \E(\Gamma)$ and that
\begin{displaymath}
  \left\|f\right\|^2 \coloneq
  \sum_{e \in \E(\Gamma)}\int_e |f(x)|^2 \, dx < \infty.
\end{displaymath}
Similarly, the Sobolev space $\sobspace 1 \Gamma$ is defined as
containing the continuous functions $f : \Gamma \to \complex$ such that
$f \upharpoonright e \in \sobspace 1 e$ for any $e \in \E(\Gamma)$ and
such that
\begin{displaymath}
  \left\|f\right\|^2_1 \coloneq
  \sum_{e \in \E(\Gamma)}\int_e \left(|f'(x)|^2 + |f(x)|^2\right) \,dx
  < \infty.
\end{displaymath}
The \emph{Neumann Laplacian} $-\Delta_\Gamma$ is the (unique) self-adjoint
operator on $\lspace 2 \Gamma$ associated with the closed quadratic form
\begin{displaymath}
  h_\Gamma[f] = \sum_{e \in \E(\Gamma)} \int_e |f'(x)|^2 \, dx
\end{displaymath}
with domain $\qdom{h_\Gamma} = \sobspace 1 \Gamma$.

\section{Reduction to an Operator on the Half-Line}
\label{reduction}

As before, $\Gamma$ is a fixed, regular, $d$-dimensional metric tree,
where $1 < d \leq 2$. We will be interested in the Schr\"odinger operator
$-\Delta_\Gamma - \alpha \tilde V$ for $\alpha > 0$.
It turns out
that bottom of the spectrum of $-\Delta_\Gamma - \alpha \tilde V$ is
described by studying a certain Schr\"odinger operator on the
half-line. Recall that by definition, there are constants $C_1, C_2 >
0$ such that
\begin{displaymath}
  C_1(1+t)^{d-1} \leq g_\Gamma(t) \leq C_2 (1+t)^{d-1}.
\end{displaymath}
Let
\begin{displaymath}
  E^+ = \frac{C_1}{C_2}
  \quad \textrm{ and } \quad
  E^- = \frac{C_2}{C_1}.
\end{displaymath}

\begin{lem} \label{lem:infspecsame}
  Let $V$ be a bounded, non-negative and measurable function on
  $(0,\infty)$. Define a function $\tilde V$ on $\Gamma$ by
  \begin{displaymath}
    \tilde V(x) = V(|x|), \quad x \in \Gamma,
  \end{displaymath}
  and the Hilbert space
  $\hilbert = \lspace 2 {(0,\infty), (1+t)^{d-1} \, dt}$.
  For $\alpha > 0$, let $H_\alpha^\pm$ be the self-adjoint operator in
  $\hilbert$ associated with the closed quadtratic form
  \begin{displaymath}
    h_\alpha^\pm[u] = \int_0^\infty \left(
      \left|u'(t)\right|^2 - \alpha E^\pm V(t)\left|u(t)\right|^2
    \right) (1+t)^{d-1} \, dt,
  \end{displaymath}
  with domain
  $\qdom{h_\alpha^\pm} = \sobspace 1 {(0,\infty), (1+t)^{d-1}\,dt}$.
  Then, for any $\alpha > 0$ for which $\inf \spec {H_\alpha^+} < 0$,
  \begin{displaymath}
    \inf \spec{H_\alpha^-}
    \leq \inf \spec{-\Delta_\Gamma - \alpha \tilde V} 
    \leq \inf \spec{H_\alpha^+}.
  \end{displaymath}
\end{lem}
\begin{rem}
  If, as in Theorem \ref{thm:main}, $V(t) \approx 1/(1+t)^\gamma$, where
  $1 < \gamma \leq d \leq 2$, $\gamma \neq 2$, then Lemma \ref{lem:testfunc}
  below shows that $\inf \spec{H_\alpha^+} < 0$ as $\alpha \to 0+$.
  Therefore the lemma is applicable.
\end{rem}
\begin{proof}
  For $\alpha > 0$, define the quadratic form
  $a_\alpha$ in $\lspace 2 {(0,\infty),
  g_\Gamma(t) \, dt}$ by
  \begin{displaymath}
    a_\alpha[u] = \int_0^\infty \left(
      |u'(t)|^2 - \alpha V(t)|u(t)|^2
    \right) \, g_\Gamma(t) \, dt,
  \end{displaymath}
  with domain $\qdom{a_\alpha} = \sobspace 1 {(0,\infty), g_\Gamma}$.
  Let $A_\alpha$ be the
  self-adjoint operator associated with $a_\alpha$. By using the
  standard orthogonal decomposition of the operator $-\Delta_\Gamma -
  \alpha \tilde V$, as described in
  \cite{naimark-solomyak,naimark-solomyak2} 
  and the arguments of \cite[Sec.~5.3]{kov},
  it is seen that
  \begin{equation}
    \inf \spec {-\Delta_\Gamma - \alpha \tilde V}
    = \inf \spec {A_\alpha}.
  \end{equation}

  The following is a variant of the techniques employed in \cite{kov}.
  Assume that $\alpha > 0$ is such that $\inf \spec {H_\alpha^+} <
  0$. Choose an arbitrary $\epsilon$ with
  $0 < \epsilon < \left|\inf \spec{H_\alpha^+}\right|$.
  Let $u \in \qdom {h_\alpha^+}$ be such that
  \begin{displaymath}
    \frac{h_\alpha^+[u]}{\int_0^\infty |u(t)|^2 (1+t)^{d-1} \, dt}
    \leq \inf \spec {H_\alpha^+} + \epsilon.
  \end{displaymath}
  Now, $u \in \qdom{a_\alpha}$ and
  \begin{displaymath}
    a_\alpha[u] \leq
    C_2 h_\alpha^+[u].
  \end{displaymath}
  In particular,
  \begin{displaymath}
    h_\alpha^+[u] < 0
    \quad \textrm{ and } \quad
    a_\alpha[u] < 0.
  \end{displaymath}
  We also have that
  \begin{displaymath}
    \int_0^\infty |u(t)|^2 g_\Gamma(t) \, dt
    \leq C_2 \int_0^\infty |u(t)|^2 (1+t)^{d-1} \, dt,
  \end{displaymath}
  and therefore,
  \begin{displaymath}
    \inf \spec {A_\alpha}
    \leq \frac{a_\alpha[u]}{\int_0^\infty |u(t)|^2 g_\Gamma(t) \, dt}
    \leq \frac{h_\alpha[u]}{\int_0^\infty |u(t)|^2 (1+t)^{d-1} \, dt}
    \leq \inf \spec {H_\alpha^+} + \epsilon.
  \end{displaymath}
  Since $\epsilon$ was arbitrary, it follows that
  \begin{displaymath}
    \inf \spec {A_\alpha}
    \leq \inf \spec {H_\alpha^+}.
  \end{displaymath}
  Similarly, it is shown that
  \begin{displaymath}
    \inf \spec {H_\alpha^-} \leq \inf \spec {A_\alpha}. \qedhere
  \end{displaymath}
\end{proof}

\section{Estimates of the Bottom of the Spectrum}
\label{sec:bottomspectrum}

Consider the space
$\hilbert = \lspace{2}{(0,\infty) ; (1+x)^{d-1}\,dx}$ and a
measurable, bounded potential
$V$. For $\alpha > 0$, let
$H_\alpha$ be the operator in $\hilbert$
determined by the closed quadratic form
\begin{displaymath}
  h_\alpha[u]  = \int_0^\infty \left(
    \left|u'(x)\right|^2 - \alpha V(x)\left|u(x)\right|^2
    \right) (1+x)^{d-1}  \, dx
\end{displaymath}
with domain
\begin{displaymath}
  \qdom{h_\alpha} = \left\{ u \in \hilbert \st u' \in \hilbert \right\}.
\end{displaymath}
Here, $d \in (1,2]$ is the dimension of the underlying tree.

\begin{rem}
  It can be seen by the standard arguments, see
  e.g.~\cite{kov}, that under the above conditions on the potential
  $V$, the essential spectrum of $H_\alpha$ is
  $[0,\infty)$.
  Hence
  the negative spectrum of $H_\alpha$ can only contain eigenvalues of
  finite multiplicity. Furthermore,
  \begin{displaymath}
    \inf \spec{H_\alpha} \to 0, \quad \alpha \to 0+.
  \end{displaymath}
\end{rem}

\begin{thm} \label{thm:estimates}
  Suppose that $V$ is measurable and that there are positive constants
  $C_1$ and $C_2$ such that for any $x > 0$,
  \begin{displaymath}
    \frac{C_1}{(1+x)^\gamma} \leq V(x) \leq \frac{C_2}{(1+x)^\gamma}.
  \end{displaymath}
  \begin{enumerate}[{\normalfont (i)}]
  \item
    If $1 < \gamma < d \leq 2$, then there are constants
    $D_1, D_2 > 0$ such that
    \begin{displaymath}
      -D_1\alpha^{\frac{2}{2-\gamma}}
      < \inf \spec{H_\alpha}
      < -D_2\alpha^{\frac{2}{2-\gamma}},
      \quad \alpha \to 0+.
    \end{displaymath}
  \item
    If $1 < \gamma = d < 2$, then there are constants
    $D_1, D_2 > 0$ such that
    \begin{displaymath}
      -D_1 \left|\alpha \log \alpha\right|^{\frac{2}{2-\gamma}}
      < \inf \spec{H_\alpha}
      < -D_2 \left|\alpha \log \alpha\right|^{\frac{2}{2-\gamma}},
      \quad \alpha \to 0+.
    \end{displaymath}
  \end{enumerate}
\end{thm}

The proof of Theorem
\ref{thm:estimates} relies on the following two lemmas, of which the
first is proved in Appendix \ref{sec:birmanschwinger}.

\begin{lem} \label{lem:trineq}
  Suppose that $1 < \gamma \leq d \leq 2$, $\gamma \neq 2$,
  and that there is a $C > 0$
  such that $V$ satisfies
  \begin{displaymath}
    0 \leq V(x) \leq \frac{C}{(1+x)^\gamma}.
  \end{displaymath}
  Then there is a constant $D > 0$ such that for any $E > 0$,
  there is a non-negative
  trace-class operator $Q_E$ on $\ltwo$ whose trace satisfies
  \begin{displaymath}
    \tr Q_E \leq \left\{
      \begin{array}{ll}
        DE^\frac{\gamma-2}{2},
        & \quad \gamma < d, \\
        DE^{\frac{\gamma-2}{2}}\left(1+\left|\log E\right|\right),
        & \quad \gamma = d,
      \end{array}
    \right.
  \end{displaymath}
  and such that $\alpha^{-1}$ is an eigenvalue of $Q_E$ if and only
  if $-E$ is an eigenvalue of the operator $H_\alpha$.
\end{lem}

\begin{lem} \label{lem:testfunc}
  Suppose that $V$ is bounded and that there is a $C > 0$ such that
  \begin{displaymath}
    V(x) \geq \frac{C}{(1+x)^\gamma}.
  \end{displaymath}
  \begin{enumerate}[{\normalfont (i)}]
  \item If $1 < \gamma < d \leq 2$, then there is a
    $D > 0$ such that
    \begin{displaymath}
      \inf \spec{H_\alpha} < -D\alpha^{\frac{2}{2-\gamma}},
      \quad \alpha \to 0+.
    \end{displaymath}
  \item If $1 < \gamma = d < 2$, then there is a $D > 0$ such that
    \begin{displaymath}
      \inf \spec{H_\alpha}
      < -D\left|\alpha \log \alpha\right|^{\frac{2}{2-\gamma}},
      \quad \alpha \to 0+.
    \end{displaymath}
  \end{enumerate}
\end{lem}
\begin{proof} \text{ }
  \begin{enumerate}[{\normalfont (i)}]
  \item For the first case, assume that $1 < \gamma < d \leq 2$.
    Choose $\alpha > 0$ and $\delta$ with
    $0 < \delta < 1$.
    Consider the function
    \begin{displaymath}
      u_\delta(x) = e^{-\delta x}.
    \end{displaymath}
    Clearly $u_\delta \in \qdom{h_\alpha}$ and
    \begin{equation} \label{eq:halphaest}
      \begin{split}
        h_\alpha\left[u_\delta\right]
        & \leq
        \int_0^\infty \left(
          \left|u_\delta'(x)\right|^2(1+x)^{d-1}
          - C \alpha \left|u_\delta(x)\right|^2(1+x)^{d-\gamma-1}
        \right) \, dx \\
        & =
        \delta^{2-d}e^{2\delta}\int_\delta^\infty e^{-2x}x^{d-1} \, dx
        - \alpha \delta^{\gamma-d}e^{2\delta}
        C \int_\delta^\infty e^{-2x} x^{d-\gamma-1} \, dx \\
        & \leq \delta^{-d}e^{2\delta} \left(
          \delta^{2}\int_0^\infty e^{-2x}(1+x) \, dx
          - \alpha \delta^{\gamma} C \int_1^\infty \frac{e^{-2x}}{x} \,
          dx
        \right).
      \end{split}
    \end{equation}
    Furthermore,
    \begin{equation} \label{eq:udeltanormest}
      \begin{split}
        \left\|u_\delta\right\|_{\hilbert}^2
        & = \delta^{-d}e^{2\delta}\int_\delta^\infty e^{-2x}x^{d-1}\,dx \\
        & \leq \delta^{-d}e^{2\delta}\int_0^\infty e^{-2x}x^{d-1}\,dx \\
        & = \delta^{-d}e^{2\delta}2^{-d}\, \Gamma(d).
      \end{split}
    \end{equation}
    Now choose a constant $K > 0$ such that
    \begin{displaymath}
      \widetilde{K}_\gamma \coloneq K^\gamma C \int_1^\infty
      \frac{e^{-2x}}{x}\,dx
      - K^2\int_0^\infty e^{-2x}(1+x)\,dx > 0.
    \end{displaymath}
    Assume that $\alpha > 0$ is small enough to ensure that
    \begin{displaymath}
      \delta \coloneq K\alpha^{\frac{1}{2-\gamma}} < 1.
    \end{displaymath}
    Then we have that
    \begin{displaymath}
      \alpha \delta^\gamma = K^\gamma \alpha^{\frac{2}{2-\gamma}}
    \end{displaymath}
    and that
    \begin{displaymath}
      \delta^2 = K^2 \alpha^{\frac{2}{2-\gamma}}.
    \end{displaymath}
    Therefore, by \eqref{eq:halphaest},
    \begin{equation} \label{eq:halphaest2}
      h_\alpha[u_\delta] \leq - \delta^{-\delta}e^{2\delta}
      \widetilde{K}_\gamma
      \alpha^{\frac{2}{2-\gamma}}.
    \end{equation}
    In particular $h_\alpha[u_\delta] < 0$ and by combining
    \eqref{eq:udeltanormest} with \eqref{eq:halphaest2} we obtain
    \begin{displaymath}
      \inf \spec{H_\alpha} 
      \leq \frac{h_\alpha[u_\delta]}{\|u_\delta\|_\hilbert^2}
      \leq - \frac{2^d \widetilde{K}_\gamma}{\Gamma(d)}
      \alpha^{\frac{2}{2-\gamma}}.
    \end{displaymath}

    \item To prove the second statement, let $1 < \gamma = d < 2$.
      Let $\beta > 2$ be large enough to guarantee that
      \begin{displaymath}
        M \coloneq \frac{C}{8(2-d)}\log \frac{\beta}{2}
        - \frac{2^d}{d}\beta^{d-2} > 0
      \end{displaymath}
      Choose any $\alpha$ with $0 < \alpha < 1$ small enough to satisfiy
      \begin{displaymath}
        \nu \coloneq
        \frac{1}{\left| \alpha \log \alpha \right|^{\frac{1}{2-d}}}
        \geq e.
      \end{displaymath}
      Note that $x \log x \geq -1/2$ for any
      $x > 0$. Hence
      \begin{displaymath}
        y - \log y = \left(1+y^{-1}\log y^{-1}\right)y \geq \frac{1}{2}y
      \end{displaymath}
      for any $y > 0$ and therefore, since $0 < \alpha < 1$,
      \begin{equation} \label{eq:alogn}
        \alpha \log \nu
        = \frac{1}{2-d}\alpha\left(|\log \alpha| - \log |\log \alpha|\right)
        \geq \frac{1}{2(2 - d)}\left|\alpha \log \alpha\right|.
      \end{equation}
      Let
      \begin{displaymath}
        \mu = \beta \nu.
      \end{displaymath}
      Note that $\mu > 2e$ and define
      \begin{displaymath}
        w(x) = \left\{
          \begin{array}{ll}
            1 - \frac{x}{\mu}, \quad & 0 < x < \mu, \\
            0, & x \geq \mu.
          \end{array}
        \right.
      \end{displaymath}
      Clearly $w \in \qdom{h_\alpha}$ and since $\mu \geq 1$,
      \begin{equation} \label{eq:kinetic}
        \begin{split}
          \int_0^\infty\, \left|w'(x)\right|^2 (1+x)^{d-1} \, dx
          & = \frac{1}{d\mu^2}\left((1+\mu)^d - 1\right) \\
          & \leq \frac{1}{d\mu^2}(1+\mu)^d \\
          & \leq \frac{2^d}{d} \mu^{d-2}.
        \end{split}
      \end{equation}
      On the other hand,
      using the assumption on $V$ we get a lower bound on the potential
      energy by
      \begin{equation} \label{eq:potential}
        \begin{split}
          \alpha \int_0^\infty |w(x)|^2 V(x) (1+x)^{d-1} \, dx
          &\geq C\alpha \int_0^\mu
          \frac{\left|1-\frac x \mu\right|^2}{1+x} \, dx \\
          &\geq C\alpha \int_0^{\mu/2}
          \frac{\left|1-\frac x \mu\right|^2}{1+x} \, dx \\
          & \geq \frac{C\alpha}{4} \int_0^{\mu/2}
          \frac{1}{1+x} \, dx \\
          & = \frac{C\alpha}{4} \log \left(1 + \frac \mu 2\right) \\
          & \geq \frac{C\alpha}{4} \log \frac \mu 2.
        \end{split}
      \end{equation}
      Using \eqref{eq:kinetic}, \eqref{eq:potential} and
      \eqref{eq:alogn} we get that
      \begin{displaymath}
        \begin{split}
          h_\alpha[w]
          & \leq \frac{2^d}{d}\mu^{d-2} - \frac{C\alpha}{4} 
          \log \frac \mu 2 \\
          & = \frac{2^d}{d}\beta^{d-2}\nu^{d-2}
          - \frac{C}{4}\log\frac \beta 2 \cdot \alpha \log \nu \\
          & \leq |\alpha \log \alpha|
          \left(
            \frac{2^d}{d}\beta^{d-2} - \frac{C}{4} \log \frac \beta 2
            \cdot \frac{1}{2(2-d)}
          \right) \\
          & = - M |\alpha \log \alpha|.
        \end{split}
      \end{displaymath}
      Finally, since $|w| \leq 1$,
      \begin{displaymath}
        \begin{split}
          \|w\|^2_\hilbert
          & \leq \int_0^\mu (1+x)^{d-1} \, dx  \leq \frac{2^d}{2}\mu^d \\
          & = \frac{2^d}{d} \beta^d |\alpha \log \alpha|^{-\frac{d}{2-d}}.
        \end{split}
      \end{displaymath}
      It thus follows that
      \begin{displaymath}
        \inf \spec {H_\alpha} 
        \leq \frac{h_\alpha[w]}{\|w\|_\hilbert^2}
        \leq - \frac{dM}{(2\beta)^d}
        |\alpha \log \alpha|^{\frac{2}{2-d}}. \qedhere
      \end{displaymath}
    \end{enumerate}
  \end{proof}

\begin{proof}[Proof of Theorem \ref{thm:estimates}]
  Start by assuming that $1 < \gamma \leq d \leq 2$, $\gamma \neq 2$.
  Let
  \begin{displaymath}
    E(\alpha) = -\inf \spec{H_\alpha}
  \end{displaymath}
  for every $\alpha > 0$. By Lemma
  \ref{lem:testfunc}, $E(\alpha) > 0$. Since the negative
  spectrum of $H_\alpha$ is discrete, $-E(\alpha)$ is an eigenvalue of
  $H_\alpha$. By Lemma \ref{lem:trineq} there is a $D = D(\gamma,d) >
  0$ such that for any $\alpha > 0$, there is a non-negative
  trace-class operator $Q_{E(\alpha)}$ whose trace is estimated by
  \begin{displaymath}
    \tr Q_{E(\alpha)} \leq \left\{
      \begin{array}{ll}
        D\cdot\left(E(\alpha)\right)^{\frac{\gamma-2}{2}},
        & \quad \gamma < d, \\
        D\cdot\left(E(\alpha)\right)^{\frac{\gamma-2}{2}}
        \left(1+\left|\log E(\alpha)\right|\right),
        & \quad \gamma = d,
      \end{array}
    \right.
  \end{displaymath}
  and such that $\alpha^{-1}$ is an eigenvalue of
  $Q_{E(\alpha)}$. Since $Q_{E(\alpha)}$ is non-negative, we get that
  $\alpha^{-1} \leq \tr Q_{E(\alpha)}$ and therefore
  \begin{equation} \label{eq:ev-est}
    \frac 1 \alpha \leq
    \left\{
      \begin{array}{ll}
        D\cdot\left(E(\alpha)\right)^{\frac{\gamma-2}{2}},
        & \quad \gamma < d, \\
        D\cdot\left(E(\alpha)\right)^{\frac{\gamma-2}{2}}
        \left(1+\left|\log E(\alpha)\right|\right),
        & \quad \gamma = d.
      \end{array}
    \right.
  \end{equation}

  \begin{enumerate}[{\normalfont (i)}]
  \item
    Now consider the first case, where $1 < \gamma < d \leq 2$. Choose
    $\alpha > 0$. By \eqref{eq:ev-est},
    $D\alpha \geq \left(E(\alpha)\right)^{(2-\gamma)/2}$.
    From the fact that $2 - \gamma > 0$ it follows that
    \begin{displaymath}
      E(\alpha) \leq (D\alpha)^{\frac{2}{2-\gamma}}
    \end{displaymath}
    for any $\alpha > 0$.
    To complete the proof of the first statement it now remains
    to apply Lemma \ref{lem:testfunc}.

  \item
    Let us turn to the second case, where $1 < \gamma = d < 2$.
    Introduce the function $W \in \cinf$, defined as the
    inverse of the function $y \mapsto ye^y, y > 0$. The function $W$ is
    sometimes called the Lambert W-function.
    Since $W$ is increasing, we get that if $y,z > 0$ are such that
    $z \leq ye^y$, then
    \begin{equation} \label{eq:lambertineq}
      y = W(ye^y) \geq W(z).
    \end{equation}

    For simplicity, let $\alpha_0 > 0$ be such that $E(\alpha) < 1/2$
    for any $\alpha$ with
    $0 < \alpha < \alpha_0$. Then, by \eqref{eq:ev-est} there is a
    constant $\tilde D = \tilde D(\gamma)$ such that
    \begin{equation} \label{eq:alphaest-case2}
      \frac{1}{\tilde D \alpha}
      \leq -\left(E(\alpha)\right)^{\frac{\gamma-2}{2}}\log E(\alpha)
    \end{equation}
    for any $\alpha$ with $0 < \alpha < \alpha_0$.
    Let $y(\alpha) = \log
    \left(\left(E(\alpha)\right)^{\frac{\gamma-2}{2}}\right)$.
    Then by \eqref{eq:alphaest-case2},
    \begin{displaymath}
      \frac{1}{\tilde D \alpha}
      \leq \frac{2}{2-\gamma} y(\alpha)e^{y(\alpha)}
    \end{displaymath}
    whenever $0 < \alpha < \alpha_0$.
    By \eqref{eq:lambertineq},
    \begin{displaymath}
      y(\alpha) \geq W\left(\frac{2-\gamma}{2\tilde D \alpha}\right)
    \end{displaymath}
    and therefore, since $W(x)e^{W(x)} = x$ for any $x > 0$,
    \begin{equation} \label{eq:Eest-lambert}
      \begin{split}
        E(\alpha)
        & \leq \left(
          \exp\left(
            -W\left(\frac{2-\gamma}{2\tilde D \alpha} \right)
          \right)
        \right)^{\frac{2}{2-\gamma}} \\
        & = \left(
          \frac{
            W\left(\frac{2-\gamma}{2\tilde D \alpha}\right)
          }{
            W\left(\frac{2-\gamma}{2\tilde D \alpha}\right)
            \exp\left(
              W\left(\frac{2-\gamma}{2\tilde D \alpha}\right)
            \right)
          }
        \right)^{\frac{2}{2-\gamma}} \\
        & = \left(
          \frac{2\tilde D}{2-\gamma}\alpha
          W\left(\frac{2-\gamma}{2\tilde D \alpha}\right)
        \right)^{\frac{2}{2-\gamma}}
      \end{split}
    \end{equation}
    for such $\alpha$. Now, note that for any $x > 0$,
    $W(1/x)\exp(W(1/x)) = 1/x$. Therefore
    $\log(1/x) = \log W(1/x) + W(1/x)$ and since $W(1/x) \to \infty$
    as $x \to 0+$ we get that
    \begin{displaymath}
      \log \frac 1 x \sim W\left(\frac 1 x\right),
      \quad
      x \to 0+.
    \end{displaymath}
    In particular, letting $x = 2\tilde D \alpha/(2-\gamma)$, we get that
    \begin{displaymath}
      W\left(\frac{2-\gamma}{2\tilde D \alpha}\right)
      \leq K \log \frac{2-\gamma}{2\tilde D \alpha},
      \quad \alpha \to 0+,
    \end{displaymath}
    for any fixed $K > 1$.
    Together with \eqref{eq:Eest-lambert} this gives
    \begin{displaymath}
      E(\alpha) \leq
      \left(
        \frac{2K\tilde D}{2-\gamma}\alpha
        \left(
          \left|\log \alpha\right|
          +
          \log\left(\frac{2-\gamma}{2\tilde D}\right)
        \right)
      \right)^{\frac{2}{2-\gamma}},
      \quad
      \alpha \to 0+.
    \end{displaymath}
    Applying Lemma \ref{lem:testfunc} completes the proof of the
    second case. \qedhere
  \end{enumerate}
\end{proof}

\appendix
\section{Appendix: The Fourier-Bessel Transform}
\label{sec:fourierbessel}

\subsection{Properties of the Bessel functions}
Let $\nu$ be a real number. Denote by $J_\nu$ and
$Y_\nu$ the Bessel functions of the first and second type,
respectively. Also, let
\begin{displaymath}
  H_\nu^{(1)} = J_\nu + i Y_\nu
  \quad \textrm{ and } \quad
  H_\nu^{(2)} = J_\nu - i Y_\nu
\end{displaymath}
denote the Hankel functions. We will use the following properties of
these functions, as listed in Chapter 9 of \cite{stegun-abramowitz}:
\begin{prop} \label{prop:besselprop}
  Let $\nu \in \reals$.
  The functions $J_\nu$, $Y_\nu$, $H^{(1)}_\nu$ and $H^{(2)}_\nu$ are
  analytic in $\complex \setminus (-\infty,0]$ and satisfy the
  following:
  \begin{enumerate}[{\normalfont (i)}]
  \item The functions $J_\nu$ and $Y_\nu$ are linearly independent.

  \item
    Let $C_\nu$ denote $J_\nu$, $Y_\nu$, $H^{(1)}_\nu$,
    $H^{(2)}_\nu$ or any linear combination of these functions. Then
    we have that
    \begin{displaymath}
      C_\nu'(z) = C_{\nu-1}(z) - \frac{\nu}{z}C_\nu(z)
      \quad \textrm{ and } \quad
      C_\nu'(z) = -C_{\nu+1}(z) + \frac{\nu}{z}C_\nu(z).
    \end{displaymath}

  \item If $\nu > 0$ and $z \to 0$, then
    \begin{displaymath}
      J_\nu(z) \textrm{ is bounded}
      \quad \textrm{ and } \quad
      Y_\nu(z) \sim -\frac 1 \pi \Gamma(\nu)\left(\frac z 2\right)^{-\nu}.
    \end{displaymath}

  \item If $\nu = 0$ and $z \to 0$, then
    \begin{displaymath}
      J_\nu(z) \textrm{ is bounded}
      \quad \textrm{ and } \quad
      Y_\nu(z) \sim -\frac 2 \pi \log z.
    \end{displaymath}

  \item If $x$ is real and $x \to \infty$, then
    \begin{displaymath}
      J_\nu(x) \sim \sqrt{\frac{2}{\pi x}}
      \cos\left(x \!-\! \frac {\nu \pi}{2} \!-\! \frac{\pi}{4}\right)
      \quad \textrm{ and } \quad
      Y_\nu(x) \sim \sqrt{\frac{2}{\pi x}}
      \sin\left(x \!-\! \frac {\nu \pi}{2} \!-\! \frac{\pi}{4}\right).
    \end{displaymath}

  \item If $z \in \complex \setminus (-\infty,0]$ and $z \to \infty$, then
    \begin{displaymath}
      H_\nu^{(1)}(z) \sim \sqrt{\frac{2}{\pi z}} e^{i(z - \nu\pi/2-\pi/4)}
      \quad \textrm{ and } \quad
      H_\nu^{(2)}(z) \sim \sqrt{\frac{2}{\pi z}} e^{-i(z - \nu\pi/2-\pi/4)}.
    \end{displaymath}

  \item We have that $\overline{J_\nu(z)} = J_\nu(\bar z)$ and
    $\overline{Y_\nu(z)} = Y_\nu(\bar z)$.

  \item The following Wronskian formula holds:
    \begin{displaymath}
      J_{\nu + 1}(z)Y_\nu(z) - J_\nu(z)Y_{\nu+1}(z) = 2/(\pi z).
    \end{displaymath}
  \end{enumerate}
\end{prop}

\subsection{The Weyl-Titchmarsch-Kodaira Theorem}
Let $\tau$ be a formally self-adjoint formal
differential operator of order $n$ on the interval $(0,\infty)$, and
let $T$ be a self-adjoint realization of $\tau$ in $\lspace 2
{0,\infty}$. We will assume that $T \geq 0$. For such operators it is
sometimes possible to give an explicit description of the spectral measure,
using the following technique from Chapter XIII of \cite{dunford-schwartz}:

\begin{prop} \label{prop:wtk}
  Let $U$ be a fixed open neighborhood of $(0,\infty)$
  and let the functions
  $\sigma_1(\cdot,\lambda),\ldots,\sigma_n(\cdot,\lambda)$
  be continuous on $(0,\infty) \times U$, analytically
  dependent on $\lambda$ for $\lambda \in U$, and
  form a basis for the solutions of the equation
  \begin{displaymath}
    \tau\sigma = \lambda\sigma,
    \quad \lambda \in U.
  \end{displaymath}
  Suppose that for $\lambda \in U \setminus [0,\infty)$, the resolvent
  $(T-\lambda)^{-1}$ is an integral operator with kernel $K_\lambda$
  that satisfies
  \begin{displaymath}
    K_\lambda(x,y)
    = \theta(\lambda)\sigma_1(x,\lambda)
    \overline{\sigma_1(y,\overline \lambda)}
    + \sum_{i,j=1}^n a_{ij}\sigma_i(x,\lambda)
    \overline{\sigma_j(y,\overline \lambda)}
    \quad \textrm{ if } y < x,
  \end{displaymath}
  for some function $\theta$ and complex numbers $a_{ij}$.
  Then $\theta$ is analytic in $U \setminus [0,\infty)$ and
  \begin{displaymath}
    \rho(a,b) = \lim_{\delta \to 0} \lim_{\epsilon \to 0+}
    \frac{1}{2\pi i} \int_{a+\delta}^{b-\delta}
    \left(\theta(x + i\epsilon) - \theta(x - i\epsilon)\right) \,
    dx,
    \quad 0 < a < b,
  \end{displaymath}
  defines a positive Borel measure $\rho$ on
  $(0,\infty)$ that satisfies:

  \begin{enumerate}[{\normalfont (i)}]
  \item There is an isometric
    isomorphism $\vtrans$ from $\lspace 2 {0,\infty}$
    onto $\lspace 2 {(0,\infty),\rho}$, given by
    \begin{displaymath}
      (\vtrans \varphi)(y)
      = \int_0^\infty \varphi(x)\overline{\sigma_1(x,y)} \, dx,
      \quad
      \supp \varphi \Subset (0,\infty).
    \end{displaymath}

    \item The inverse of $\vtrans$ is given by
      \begin{displaymath}
        (\vtrans^{-1} f)(x)
        = \int_0^\infty f(y)\sigma_1(x,y)\,d\rho(y),
        \quad \supp f \Subset (0,\infty).
      \end{displaymath}

    \item We have that
      \begin{displaymath}
        \vtrans \dom T = \left\{
          f(y) \st yf(y) \in \lspace 2 {(0,\infty),\rho}
        \right\}
      \end{displaymath}
      and
      \begin{displaymath}
        (\vtrans T \vtrans^{-1} f)(y) = yf(y),
        \quad
        f \in \vtrans \dom T.
      \end{displaymath}
  \end{enumerate}

\end{prop}

\subsection{An Auxiliary Schr\"odinger Operator on the Half-Line}
\label{sec:auxop}

Let $1 < d \leq 2$ and consider the self-adjoint operator
\begin{displaymath}
  H_0 = -\frac{d^2}{dx^2} + \frac{(d-1)(d-3)}{4(1+x)^2}
\end{displaymath}
in the space $\lspace 2 {0,\infty}$ defined by the introduction of the
boundary condition
\begin{displaymath}
  \varphi'(0) = \frac{d-1}{2} \varphi(0).
\end{displaymath}

\begin{lem} \label{lem:equivops}
  Let the space $\hilbert$, the operator $H_\alpha$ and the potential
  $V$ be as in Section \ref{sec:bottomspectrum}.
  Then $H_\alpha$ is unitarily equivalent to $H_0 - \alpha V$ for any
  $\alpha > 0$.
\end{lem}
\begin{proof}
  Introduce the isometric isomorphism $\wtrans$ from
  $\lspace 2 {0,\infty}$ onto $\hilbert$ defined by
  \begin{displaymath}
    (\wtrans\varphi)(x) = \varphi(x)(1+x)^{(1-d)/2}, \quad
    \varphi \in \lspace 2 {0,\infty}.
  \end{displaymath}
  Recall that $h_\alpha$ is the closed quadratic form corresponding to the
  operator $H_\alpha$.
  Choose $\varphi,\psi \in \dom{H_0}$ and let $u = \wtrans\varphi$ and $v =
  \wtrans\psi$. Clearly $u,v \in \qdom{h_\alpha}$
  and partial integration gives us
  \begin{displaymath}
    h_\alpha[u,v] = -\int_0^\infty \!\!\!\! \varphi''(x)\overline{\psi(x)}\, dx
    + \int_0^\infty \!\!
    \left(\frac{(d-1)(d-3)}{4(1+x)^2} - \alpha V(x)\right)
      \varphi(x)\overline{\psi(x)} \, dx.
  \end{displaymath}
  Therefore the operator $\wtrans(H_0 - \alpha V)\wtrans^{-1}$
  is associated to the quadratic form $h_\alpha$, which proves the statement.
\end{proof}

\begin{lem} \label{lem:fbtransform}
  The transformation $\utrans : \lspace 2 {0,\infty} \to \lspace 2 {0,\infty}$
  given by
  \begin{displaymath}
    (\utrans\varphi)(p) = \int_0^\infty
    \varphi(x)\sqrt{p(1+x)}f_d(p,x)\,dx,
    \quad \supp \varphi \Subset (0,\infty),
  \end{displaymath}
  where
  \begin{displaymath}
    f_d(p,x) = \frac{
      J_{-\frac d 2}(p)Y_{\frac{2-d}{2}}(p(1+x))
      -
      Y_{-\frac d 2}(p)J_{\frac{2-d}{2}}(p(1+x))
    }{
      \left(
      \left(J_{-\frac d 2}(p)\right)^2
      +
      \left(Y_{-\frac d 2}(p)\right)^2
      \right)^{1/2}
    }
  \end{displaymath}
  is a unitary isomorphism under which $H_0$ is equivalent
  to multiplication by the function $p \mapsto p^2$. The inverse of
  $\utrans$ is given by
  \begin{displaymath}
    (\utrans^{-1}\psi)(x) = \int_0^\infty \psi(p)\sqrt{p(1+x)}f_d(p,x)\,dp,
    \quad \supp \psi \Subset (0,\infty).
  \end{displaymath}
\end{lem}
\begin{proof}
  Throughout, the complex number $\lambda$ will be chosen from a
  sufficiently small, fixed, open neighborhood in $\complex$
  of $(0,\infty)$. Let
  \begin{displaymath}
    \sigma_1(x,\lambda)
    = \sqrt{1 \!+\! x}\left(
      J_{-\frac d 2}(\lambda^{1/2})Y_{\frac {2-d}{2}}(\lambda^{1/2}(1\!+\!x))
      -
      Y_{-\frac d 2}(\lambda^{1/2})J_{\frac {2-d}{2}}(\lambda^{1/2}(1\!+\!x))
    \right)
  \end{displaymath}
  and
  \begin{displaymath}
    \sigma_2(x,\lambda) = \sqrt{1+x}J_{\frac{2-d}{2}}(\lambda^{1/2}(1+x))
  \end{displaymath}
  and
  \begin{displaymath}
    \chi(x,\lambda) = \left\{
      \begin{array}{ll}
        \sqrt{1+x}\, H^{(1)}_{\frac{2-d}{2}}(\lambda^{1/2}(1+x))
        & \textrm{ if } \impart \lambda > 0 \\
         & \\
        \sqrt{1+x}\, H^{(2)}_{\frac{2-d}{2}}(\lambda^{1/2}(1+x))
        & \textrm{ if } \impart \lambda < 0. \\
      \end{array}
    \right.
  \end{displaymath}
  Using Proposition \ref{prop:besselprop}, it is seen that the functions
  $\sigma_1(\cdot,\lambda)$, $\sigma_2(\cdot,\lambda)$ and
  $\chi(\cdot,\lambda)$ all satisfy
  the equation
  \begin{displaymath}
    -\varphi''(x)+\frac{(d-1)(d-3)}{4(1+x)^2}\varphi(x) = \lambda \varphi(x).
  \end{displaymath}
  The same Proposition also shows that
  \begin{displaymath}
    \sigma_1'(0,\lambda) = \frac{d-1}{2}\sigma_1(0,\lambda)
    \quad \textrm{ and } \quad
    \lim_{R \to \infty} \int_R^\infty|\chi(x,\lambda)|^2 \, dx = 0
  \end{displaymath}
  whenever $\impart \lambda \neq 0$. Hence, if we let
  \begin{displaymath}
    K_\lambda(x,y) = \frac{1}
    {\sigma_1'(0,\lambda)\chi(0,\lambda)-\sigma_1(0,\lambda)\chi'(0,\lambda)}
    \cdot
    \left\{
      \begin{array}{ll}
        \chi(x,\lambda)\sigma_1(y,\lambda) & \textrm{ if } y < x \\
        \sigma_1(x,\lambda)\chi(y,\lambda) & \textrm{ if } y > x, \\
      \end{array}
    \right.
  \end{displaymath}
  it follows that the resolvent $(H_0 - \lambda)^{-1}$ is an integral
  operator with kernel $K_\lambda$ for $\lambda$ with $\impart \lambda
  \neq 0$. Some calculations give us that if $y < x$ and
  $\impart \lambda > 0$ then
  \begin{displaymath}
    K_\lambda(x,y)
    = \frac{\pi}{2J_{-\frac d 2}(\lambda^{1/2})}\left(
      \frac{i \sigma_1(x,\lambda)\sigma_1(y,\lambda)}
      {J_{-\frac d 2}(\lambda^{1/2}) + iY_{-\frac d 2}(\lambda^{1/2})}
      + \sigma_2(x,\lambda)\sigma_1(y,\lambda)
    \right)
  \end{displaymath}
  and if $y < x$ and $\impart \lambda > 0$ then
  \begin{displaymath}
    K_\lambda(x,y)
    = \frac{\pi}{2J_{-\frac d 2}(\lambda^{1/2})}\left(
      \frac{-i \sigma_1(x,\lambda)\sigma_1(y,\lambda)}
      {J_{-\frac d 2}(\lambda^{1/2}) - iY_{-\frac d 2}(\lambda^{1/2})}
      + \sigma_2(x,\lambda)\sigma_1(y,\lambda)
    \right).
  \end{displaymath}
  Proposition \ref{prop:besselprop} provides that
  $\overline{\sigma_1(\cdot,\lambda)} =
  \sigma_1(\cdot,\overline \lambda)$ and that $\sigma_1$ and $\sigma_2$
  are linearly independent.
  Hence, applying Proposition \ref{prop:wtk}, we get
  that the operator $\vtrans$ defined by
  \begin{displaymath}
    (\vtrans \varphi)(y) = \int_0^\infty \varphi(x)\sigma_1(x,y) \, dx,
    \quad \supp \varphi \Subset (0,\infty)
  \end{displaymath}
  is an isometric isomorphism from $\lspace 2 {0,\infty}$ onto
  $\lspace 2 {(0,\infty), \rho}$ such that the operator $H_0$ is
  equivalent to multiplication by the identity function under this
  isomorphism. Here, the measure $\rho$ is given by
  \begin{displaymath}
    d\rho(y) = \frac{dy}{2\left(
        \left(J_{-\frac d 2}(\sqrt y)\right)^2
        + \left(Y_{-\frac d 2}(\sqrt y)\right)^2
      \right)}.
  \end{displaymath}
  Furthermore, the inverse of $\vtrans$ is given by
  \begin{displaymath}
    (\vtrans^{-1} f)(x) = \int_0^\infty
    f(y)\sigma_1(x,y) \, d\rho(y),
    \quad \supp f \Subset (0,\infty).
  \end{displaymath}
  It remains to set $\utrans = \wtrans\vtrans$, where the isometry
  $\wtrans$ from $\lspace 2 {(0,\infty),\rho}$ onto
  $\lspace 2 {0,\infty}$ is given by
  \begin{displaymath}
    (\wtrans f)(p) = f(p^2)\cdot\frac{\sqrt{p}}{\left(
        \left(J_{-\frac d 2}(p)\right)^2
        + \left(Y_{-\frac d 2}(p)\right)^2
      \right)^{1/2}},
    \quad
    f \in \lspace 2 {(0,\infty),\rho}. \qedhere
  \end{displaymath}
\end{proof}

\section{Appendix: Integral Kernel of the Birman-Schwinger Operator}
\label{sec:birmanschwinger}
In this appendix, we will use the operator $H_0$ as defined in
Section \ref{sec:auxop}. And apply the Birman-Schwinger principle,
\cite{birman,schwinger}, to prove Lemma \ref{lem:trineq}.
Moreover, $d$ and $\gamma$ will be fixed
numbers satisfying $1 < \gamma \leq d \leq 2$ and $\gamma \neq 2$,
and the measurable potential function $V$ will satisfy
\begin{equation} \label{eq:Vcond}
  0 \leq V(x) \leq \frac{C}{(1+x)^\gamma},
\end{equation}
for some $C > 0$. Let $\Omega = (0,\infty) \times (0,\infty)$
and consider the function $l_E$ on $\Omega$ given by
\begin{displaymath}
  l_E(p,x) =
  f_d(p,x)
  \left(
    \frac{
      p(1+x)
      V(x)
    } {
      p^2 + E
    }
  \right)^{1/2},
\end{displaymath}
for $E > 0$. Here, as in Lemma \ref{lem:fbtransform}, we have that
\begin{displaymath}
  f_d(p,x) = \frac{
    J_{-\frac d 2}(p)Y_{\frac{2-d}{2}}(p(1+x))
    -
    Y_{-\frac d 2}(p)J_{\frac{2-d}{2}}(p(1+x))
  }{
    \left(
      \left(J_{-\frac d 2}(p)\right)^2
      +
      \left(Y_{-\frac d 2}(p)\right)^2
    \right)^{1/2}
  }.
\end{displaymath}
Also define the bounded integral operator $L_E$ by
\begin{displaymath}
  \left(L_E\psi\right)(p) = \int_0^\infty l_E(p,x)\psi(x)\,dx,
  \quad
  \psi \in \lspace 2 {0,\infty}.
\end{displaymath}
That $L_E$ is well-defined and even a Hilbert-Schmidt operator is
provided by the following lemma:

\begin{lem} \label{lem:l-est}
  Suppose that $1 < \gamma \leq d \leq 2$, $\gamma \neq 2$ and
  that $V$ satisfies \eqref{eq:Vcond}.
  Then the following holds:
  \begin{enumerate}[{\normalfont (i)}]
  \item If $1 < \gamma < d \leq 2$ then there is a constant $B_1 > 0$
    such that for any $E > 0$,
    \begin{displaymath}
      \iint_\Omega \left| l_E(p,x) \right|^2 \, dx \, dp
      < B_1 E^{\frac{\gamma-2}{2}}.
    \end{displaymath}
  \item If $1 < \gamma = d < 2$ then there is a constant $B_2 > 0$
    such that for any $E > 0$,
    \begin{displaymath}
      \iint_\Omega \left| l_E(p,x) \right|^2 \, dx \, dp
      < B_2 E^{\frac{\gamma-2}{2}} \left( 1 + \left|\log E\right| \right).
    \end{displaymath}
  \end{enumerate}
\end{lem}

\begin{proof}
  Making a change of variables twice, we get that
  \begin{equation} \label{eq:iintl}
    \begin{split}
      \iint_\Omega \left| l_E(p,x) \right|^2 \, dx \, dp
      & \leq C \iint_\Omega \frac{p}{(p^2+E)(1+x)^{\gamma-1}}
      \left(f_d(p,x)\right)^2 \,dx \,dp \\
      & = C E^{\frac{\gamma-2}{2}}
      \int_0^\infty \!\!\!\! \int_{pE^{1/2}}^\infty
      \frac{p^{\gamma-1}}{(p^2+1)x^{\gamma-1}}F_d(p,x) \, dx \, dp,
    \end{split}
  \end{equation}
  where
  \begin{displaymath}
    \begin{split}
      F_d(p,x)
      & = \left(f_d\left(pE^{1/2},\frac{x}{pE^{1/2}}-1\right)\right)^2 \\
      & = \frac{\left(J_{-\frac d 2}\left(pE^{1/2}\right)
          Y_{\frac{d-2}{2}}\left(x\right)
          - Y_{-\frac d 2}\left(pE^{1/2}\right)
          J_{\frac{d-2}{2}}\left(x\right) \right)^2}
      {\left(J_{-\frac d 2}\left(pE^{1/2}\right)\right)^2
        + \left(Y_{-\frac d 2}\left(pE^{1/2}\right)\right)^2}.
    \end{split}
  \end{displaymath}
  Recall that
  \begin{displaymath}
    J_\nu(x) \sim \sqrt{\frac{2}{\pi x}}
    \cos\left(x \!-\! \frac {\nu \pi}{2} \!-\! \frac{\pi}{4}\right)
    \quad \textrm{ and } \quad
    Y_\nu(x) \sim \sqrt{\frac{2}{\pi x}}
    \sin\left(x \!-\! \frac {\nu \pi}{2} \!-\! \frac{\pi}{4}\right).
  \end{displaymath}
  as $x \to \infty$ for any $\nu$ and therefore
  there is a $c_1 > 0$, such that for any $x \geq 1/2$ and any $d$,
  \begin{equation} \label{eq:YJlargex}
    2\left(Y_{\frac{2-d}{2}}(x)\right)^2 \leq \frac{c_1}{x}
    \quad \textrm{ and } \quad
    2\left(J_{\frac{2-d}{2}}(x)\right)^2 \leq \frac{c_1}{x}.
  \end{equation}

  Now assume that $1 < d < 2$. We have that
  $Y_\nu(x) \sim -(1/\pi)\Gamma(\nu)(x/2)^{-\nu}$
  and that $J_\nu(x)$ is bounded as $x \to 0+$, for any $\nu > 0$.
  Hence there is a constant $c_2 > 0$
  such that for any $x$ with $0 < x < 1$
  and any $d$ with $1 < d < 2$,
  \begin{displaymath}
    2\left(Y_{\frac{2-d}{2}}(x)\right)^2 \leq \frac{c_1}{x^{2-d}}
    \quad \textrm{ and } \quad
    2\left(J_{\frac{2-d}{2}}(x)\right)^2 \leq \frac{c_1}{x^{2-d}}.
  \end{displaymath}
  Combine this with \eqref{eq:YJlargex} and obtain that as soon as
  $1 < d < 2$,
  \begin{equation} \label{eq:Fdcase1}
    \begin{split}
      F_d(p,x)
      & \leq
      \frac{
        2\left(J_{-\frac{d}{2}}\left(pE^{\frac{1}{2}}\right)\right)^2
        \!\!
        \left(Y_{\frac{2-d}{2}}\left(x\right)\right)^2
        \!\! +
        2\left(Y_{-\frac{d}{2}}\left(pE^{\frac{1}{2}}\right)\right)^2
        \!\!
        \left(J_{\frac{2-d}{2}}\left(x\right)\right)^2
      }
      {
        \left(J_{-\frac{d}{2}}\left(pE^{\frac{1}{2}}\right)\right)^2
        \!\!\ +
        \left(Y_{-\frac{d}{2}}\left(pE^{\frac{1}{2}}\right)\right)^2
      }
      \\
      & \leq \frac{c_2}{x^{2-d}}\chi_{(0,1)}(x)
      + \frac{c_1}{x}\chi_{[1,\infty)}(x),
    \end{split}
  \end{equation}
  for any $x,p > 0$. If $1 < \gamma < d < 2$, we get from
  \eqref{eq:iintl} and \eqref{eq:Fdcase1} that
  \begin{displaymath}
    \iint_\Omega \left|l_E(p,x)\right|^2 \, dx \, dp
    \leq
    C E^{\frac{\gamma-2}{2}}
    \int_0^\infty\!\!\!\!
    \int_{0}^\infty
    \frac{p^{\gamma-1}}{(p^2 + 1)x^{\gamma-1}}
    F_d(p,x)\,dx\,dp
    \leq
    D_1E^{\frac{\gamma-2}{2}},
  \end{displaymath}
  where
  \begin{displaymath}
    D_1 = D_1(\gamma,d) =
    C\left(
      c_2\!\!\int_0^\infty\!\!\!\!\int_0^1
      \frac{p^{\gamma-1}\,dx\,dp}{(p^2+1)x^{1-(d-\gamma)}}
      +
      c_1\!\!\int_0^\infty\!\!\!\!\int_1^\infty
      \frac{p^{\gamma-1}\,dx\,dp}{(p^2+1)x^\gamma}
    \right)
    < \infty.
  \end{displaymath}
  Moreover, if $1 < \gamma = d < 2$, then
  \eqref{eq:iintl} and \eqref{eq:Fdcase1} give that
  \begin{displaymath}
    \begin{split}
      \frac{
        \iint_\Omega \!\! \left|l_E(p,x)\right|^2 \! dx \, dp
      }{
        CE^{\frac{\gamma-2}{2}}
      }
      & \leq
      \int_0^\infty\frac{p^{\gamma-1}}{p^2+1}
      \left(
        c_2\!\!\int_{pE^{1/2}}^1 \!\!
        \frac{dx}{x}
        +
        c_1\!\!\int_1^\infty
        \frac{dx}{x^\gamma}
      \right)
      \, dp \\
      & =
      \frac{c_1}{\gamma \! - \! 1}\int_0^\infty \!\!\!
      \frac{p^{\gamma - 1}}{p^2+1}
      \, dp
      -
      c_2\int_0^{E^{-1/2}}
      \!\!\!\!\!\! \frac{p^{\gamma-1}}{p^2+1}
      \left(\log p \!+\! \frac{1}{2}\log E\right)
      \, dp,
    \end{split}
  \end{displaymath}
  which means that
  \begin{displaymath}
    \iint_\Omega \left|l_E(p,x)\right|^2 \, dx \, dp
    \leq D_2 E^{\frac{\gamma-2}{2}}
    + D_3 E^{\frac{\gamma-2}{2}}\left|\log E\right|,
  \end{displaymath}
  where
  \begin{displaymath}
    D_2 = D_2(\gamma)
    = C \left( \frac{c_1}{\gamma \! - \! 1} \int_0^\infty
      \frac{p^{\gamma-1}}{p^2+1} \, dp
      + c_2 \int_0^\infty \frac{p^{\gamma - 1}|\log p|}{p^2+1} \, dp
    \right)
    < \infty
  \end{displaymath}
  and
  \begin{displaymath}
    D_3 = D_3(\gamma)
    = C \frac{c_2}{2} \int_0^\infty \frac{p^{\gamma - 1}}{p^2 + 1} \, dp
    < \infty.
  \end{displaymath}

  Finally, assume that $d = 2$ and that $1 < \gamma < d$. Since
  $Y_0(x) \sim -2/\pi \cdot \log x$ and $J_0(x)$
  is bounded as $x \to 0+$, there is a constant $c_3$ such that
  \begin{displaymath}
    2\left(Y_0(x)\right)^2 \leq c_3 |\log x|
    \quad \textrm{ and } \quad
    2\left(J_0(x)\right)^2 \leq c_3 |\log x|
  \end{displaymath}
  for any $x$ with $0 < x < 1/2$. Together with \eqref{eq:YJlargex}
  this gives
  \begin{equation} \label{eq:Fdcase2}
    F_2(p,x) \leq c_3|\log x|\chi_{(0,1/2)}(x)
    + \frac{c_1}{x}\chi_{[1/2,\infty)}(x),
  \end{equation}
  for any $x,p > 0$. If $1 < \gamma < d = 2$, equations
  \eqref{eq:iintl} and \eqref{eq:Fdcase2} shows that
  \begin{displaymath}
    \iint_\Omega \left|l_E(p,x)\right|^2 \, dx \, dp
    \leq C E^{\frac{\gamma-2}{2}}  \!\!
    \int_0^\infty \!\!\!\! \int_0^\infty \!\!\!\!
    \frac{p^{\gamma-1}}{(p^2 + 1)x^{\gamma-1}}F_2(p,x) \, dx \, dp
    \leq D_4 E^{\frac{\gamma-2}{2}},
  \end{displaymath}
  where
  \begin{displaymath}
    D_4 = D_4(\gamma)
    = C \int_0^\infty
      \frac{p^{\gamma-1}}{p^2+1} \, dp \left(
      c_3 \int_0^{1/2} \!\! \frac{|\log x|}{x^{\gamma-1}} \, dx
      + c_1 \int_{1/2}^\infty \frac{1}{x^\gamma} \, dx
    \right) < \infty. \qedhere
  \end{displaymath}
\end{proof}

\begin{lem} \label{lem:equiv-b-s-ops}
  Let the unitary transformation
  $\utrans : \lspace 2 {0,\infty} \to \lspace 2 {0,\infty}$ be as in
  Lemma \ref{lem:fbtransform} and $E > 0$. Assume that $1 < \gamma
  \leq d \leq 2$, $\gamma \neq 2$ and that $V$ satisfies
  \eqref{eq:Vcond}. Then
  \begin{displaymath}
    \utrans\left(
      (H_0+E)^{-1/2}V(H_0+E)^{-1/2}
    \right)\utrans^{-1}
    =
    L_EL_E^*.
  \end{displaymath}
\end{lem}
\begin{proof}
  Choose $\psi \in \ltwo$ with $\supp \psi \Subset (0,\infty)$. Note
  that
  \begin{displaymath}
    \begin{split}
      \utrans V^{1/2}(H_0 + E)^{-1/2} \utrans^{-1} \psi
      & = \utrans V^{1/2} \utrans^{-1} \utrans (H_0 + E)^{-1/2}
      \utrans^{-1} \psi \\
      & = \utrans V^{1/2} \utrans^{-1} \left(
        \frac{\psi(\cdot)}{(\cdot^2 + E)^{1/2}}
      \right) \\
      & = \utrans V^{1/2} \left(
        \int_0^\infty \frac{\psi(p)}{(p^2 + E)^{1/2}}
        \sqrt{p(1+\cdot)} f_d(p,\cdot) \, dp
      \right) \\
      & = \utrans L_E^* \psi
    \end{split}
  \end{displaymath}
  Since the operators in question are bounded, this is enough to prove
  that
  \begin{displaymath}
    \utrans V^{1/2}(H_0 + E)^{-1/2} \utrans^{-1} = \utrans L_E^*.
  \end{displaymath}
  It follows that
  \begin{displaymath}
    \begin{split}
      \utrans (H_0 + E)^{-1/2}&V(H_0 + E)^{-1/2} \utrans^{-1} \\
      & = \left(
        \utrans V^{1/2}(H_0 + E)^{-1/2} \utrans^{-1}
      \right)^*
      \left(
        \utrans V^{1/2}(H_0 + E)^{-1/2} \utrans^{-1}
      \right) \\
      & = \left(\utrans L_E^*\right)^* \utrans L_E^* \\
      & = L_E L_E^*.
      \qedhere
    \end{split}
  \end{displaymath}
\end{proof}

We are now finally in a position to prove Lemma \ref{lem:trineq}.

\begin{proof}[Proof of Lemma \ref{lem:trineq}]
  It is well-known that if $l \in \lspace 2 \Omega$ and $L$ is the
  Hilbert-Schmidt operator defined by $(L\psi)(p) = \int_0^\infty
  l(p,x)\psi(x)\,dx$, then the operator $Q = LL^*$ is
  non-negative, trace class and satisfies
  \begin{displaymath}
    \tr Q \leq \int_\Omega |l(p,x)|^2 \, dx \, dp.
  \end{displaymath}
  It remains to choose $E > 0$, set
  \begin{displaymath}
    Q_E = \left(H_0 + E\right)^{-1/2}V\left(H_0 + E\right)^{-1/2},
  \end{displaymath}
  and use Lemmas \ref{lem:equivops}, \ref{lem:l-est} and
  \ref{lem:equiv-b-s-ops}.
\end{proof}

\textbf{Acknowledgements}. Hynek Kova\v r\'{\i}k was supported by
the German Research Foundation (DFG) under Grant KO 3636/1-1. Tomas
Ekholm and Andreas Enblom were partially supported by the ESF
programme SPECT.

\bibliography{weak}
\bibliographystyle{plain}  

\end{document}